\documentclass[12pt,letterpaper]{amsart}

\usepackage{epsfig,amsmath,amsthm,amssymb,graphicx}
\usepackage[top=1in, bottom=1in, left=1in, right=1in]{geometry}
\usepackage{color}

\usepackage[dvipsnames]{xcolor}

\usepackage{tikz}
\usetikzlibrary{matrix, arrows, cd}
\usepackage{pb-diagram}

\makeatletter
\newtheorem*{rep@theorem}{\rep@title}
\newcommand{\newreptheorem}[2]{%
\newenvironment{rep#1}[1]{%
 \def\rep@title{#2 \ref{##1}}%
 \begin{rep@theorem}}%
 {\end{rep@theorem}}}
\makeatother
	
\newtheorem{proposition}{Proposition}[section]
\newtheorem{theorem}[proposition]{Theorem}

\newtheorem*{theorem*}{Theorem}
\newtheorem*{proposition*}{Proposition}
\newtheorem*{lemma*}{Lemma}
\newtheorem*{corollary*}{Corollary}

\theoremstyle{definition}
\newtheorem{definition}[proposition]{Definition}
\newtheorem{question}[proposition]{Question}

\newreptheorem{theorem}{Theorem}

\theoremstyle{remark}
\newtheorem{remark}[proposition]{Remark}

\newcommand{\bdry}{\partial}

\newcommand{\N}{\mathbb{N}}
\newcommand{\Z}{\mathbb{Z}}
\newcommand{\F}{\mathcal{F}}
\newcommand{\C}{\mathcal{C}}

\renewcommand{\int}{\operatorname{int}}

\newcommand{\lk}{\operatorname{lk}}

\newcommand{\Id}{\operatorname{Id}}

\newcommand{\smooth}{\text{sm}}

\renewcommand{\S}{\mathcal{S}}

\newcommand{\ot}{\leftarrow}
\newcommand{\onto}{\twoheadrightarrow}

\newcommand{\pref}[1]{(\ref{#1})}

\begin{document}
\title[Homology spheres and the solvable filtration]{Topological concordance of knots in homology spheres and the solvable filtration.}

\author{Christopher W.\ Davis}
\address{Department of Mathematics, University of Wisconsin--Eau Claire}
\email{daviscw@uwec.edu}
\urladdr{www.uwec.edu/daviscw}
\date{\today}


\begin{abstract}
In 2016 Levine showed that there exists a knot in a homology 3-sphere which is not smoothly concordant to any knot in $S^3$ where one allows concordances in any smooth homology cobordism.  Whether the same is true if one allows topological concordances is not known.  One might hope that such an example might be detected by the powerful filtration of knot concordance introduced by Cochran-Orr-Teichner.  We prove that this is not the case, demonstrating that for any knot in any homology sphere there is a knot in $S^3$ equivalent to the original knot modulo any term of this filtration.  Our results apply equally well to link concordance.  As an application we prove that every winding number $\pm1$ satellite operator acts bijectively on knot concordance, modulo any term of the solvable filtration.
\end{abstract}

\maketitle

\date{\today}

\section{Introduction and statement of main results.}

In \cite{Levine2016}, A.~Levine proved the surprising result that there exist knots in homology spheres which are not smoothly concordant to any knot in $S^3$, even if one allows concordances in  smooth homology cobordisms.  Doing so answered a question of Matsumoto \cite[Problem 1.31]{KirbyList}.  In this paper we consider the topological version of the same question.  We find that every knot in a homology sphere appears to be topologically concordant to some knot in $S^3$, at least to the eyes of the powerful solvable filtration due to Cochran-Orr-Teichner \cite{COT03}.   Our techniques and results apply equally well for links in homology spheres.  Levine \cite{Levine2016} also proved that there exist winding number one satellite operators, as exemplified in Figure \ref{fig:satelliteoperation} which are not bijective as maps on smooth knot concordance.  As an application, we join the main results of this paper with ideas of Ray and the author \cite{DaRa2017} to prove that modulo any term of the solvable filtration every winding number one satellite operator is bijective.

Two knots $K$ and $J $ in $S^3$ are called \emph{topologically concordant} (or just \emph{concordant}) if $K\times\{1\}$ and $J\times \{0\}$ cobound a locally flat properly embedded annulus in $S^3\times[0,1]$.  Concordance gives an equivalence relation on the set of knots in $S^3$ and we denote by $\C$ the quotient by this relation.  A knot concordant to the unknot is called \emph{slice}.   By capping the unknot with a disk it bounds we see that a knot is slice if and only if it bounds a locally flat embedded disk in $B^4$, called a \emph{slice disk}.  In the case that the annulus or disk above happens to be smooth we instead say \emph{smoothly concordant} or \emph{smoothly slice}.  The quotient of knots by smooth concordance is denoted $\C_\smooth$.

There is a natural extension of  concordance to the set of pairs $(M,K)$ with $M$ a homology sphere and $K$ a knot in $M$.  To be precise, $(M,K)$ is \emph{homology concordant} to $(N,J)$ if there is a homology cobordism (not necessarily smooth)  from $M$ to $N$ in which $K$ and $J$ cobound  a locally flat properly embedded annulus.   
A knot which is homology concordant to the unknot is called \emph{homology slice}.  
The quotient of the set of knots in homology spheres by homology concordance is denoted $\widehat\C$.  The quotient $\widehat \C_{\smooth}$ is defined analogously, by adding smoothness to the homology cobordisms and embedded annuli.    There are a natural maps $\Psi:\C\to \widehat{\C}$  and $\Psi_{\smooth}:\C_{\smooth}\to \widehat{\C}_{\smooth}$ given by $K\mapsto(S^3, K)$.  

In the smooth category $\Psi_\smooth$ is definitely not surjective.  Indeed, take a homology sphere $M$ which is not smoothly homology cobordant to $S^3$, so that in particular there does not exist a homology cobordism from $M$ to $S^3$.  For example, $M$ might be the Poincar\'e homology sphere or any other homology sphere with nonzero Rohlin invariant.  See \cite[Definition 5.7.16]{KirbyCalculus} for a brief discussion of the Rohlin invariant.  It follows immediately from the definition of $\widehat\C_{\smooth}$ that for every knot $K$ in $M$, $(M,K)$ is not smoothly homology concordant to any knot in $S^3$ meaning $(M,K)$ is not in the image of $\Psi_\smooth$.  Even more strongly, by work of Levine \cite[Theorem 1.1]{Levine2016} there exists a pair $(M,K)\in \widehat{\C}_\smooth$ for which $M$ is smoothly homology cobordant to $S^3$ and yet $(M,K)$ does not cobound a smooth annulus with any knot in $S^3$ in any smooth homology cobordism.  

The story is quite different (and less complete) in the topological category.  As a consequence of work of Freedman-Quinn \cite[Corollary 9.3C]{FQ} every homology sphere is homology cobordant to $S^3$.   Thus, the easiest obstruction to $\Psi$ being surjective fails.  The aim of this paper is to study this question.

\begin{question}\label{quest:PsiOnto}
Given a knot $K$ in a homology sphere $M$, does there exist a homology cobordism from $M$ to $S^3$ in which $K$ cobounds a locally flat embedded annulus with some knot in $S^3$?  In other words, is the map $\Psi:{\C} \to \widehat\C$ surjective? 
\end{question}

In \cite{COT03} Cochran-Orr-Teichner introduced a groundbreaking new structure on $\C$, called the \emph{solvable filtration}.  It amounts to a sequence of nested subgroups 
$$
\dots  \F_{(n+1)}\le \F_{n}\le\dots \F_1\le  \F_0\le \C.$$
  We recall the formal definition in Section \ref{sect:defn}.  Informally a knot lies in $\F_n$ (and is called  $n$-solvable) for $n$ large if that knot bounds a locally flat disk in a 4-manifold which is algebraically highly similar the 4-ball.  It is still open whether $\underset{n=0}{\overset{\infty}{\cap}} \F_n$ consists only of slice knots.  Thus, it is not known if every phenomenon of knot concordance is detected in the quotient $\C/\F_n$ for some $n$.    
  
  The definition of $\F_n$ extends easily to give a filtration of $\widehat{\C}$:
  $$
  \dots  \widehat\F_{(n+1)}\le \widehat\F_{n}\le\dots \widehat\F_1\le  \widehat\F_0\le \widehat\C.$$
  This filtration is compatable with $\F_n$  in that $\Psi[\F_n] = \Psi[\C]\cap \widehat\F_n$ so that the induced map $\Psi:\C/\F_n \to \widehat \C/\widehat\F_n$ is well-defined and injective.   We recall the precise definition in Section \ref{sect:defn} and prove some relevant properties.

If one expects that the topological setting should agree with the smooth, there should exist a knot in a homology sphere which is not homology concordant to any knot in $S^3$.  One may further hope that this is detectable by the solvable filtration, since every other currently known feature of topological concordance is detected by the solvable filtration. That is, for some $n\in \Z_{\ge0}$ one might expect that there is a knot in a homology sphere whose class in $\widehat \C/\widehat \F_n$ is not in the image of $\Psi$.  Out first main result shows that this is not the case: 

   \begin{theorem}\label{thm:mainKnot}
   Let $K$ be a knot in a homology sphere.  Then for every $n\in \Z_{\ge0}$ there exists a knot $K'$ in $S^3$ such that $(M,K)$ is equivalent to $(S^3,K')$ in $\widehat\C/\widehat \F_n$.  Thus, $\Psi:\C/\F_n \to \widehat \C/\widehat\F_n$ is surjective and so bijective.
   \end{theorem}

   Motivated by this result, we conjecture that $\Psi:\C\to \widehat \C$ is surjective and so the answer to Question \ref{quest:PsiOnto} is yes.  
   
   The techniques we use in this paper apply equally well to links in homology spheres.  In Section \ref{sect:defn} we provide a notion of $n$-solvable concordance of links.  This equivalence relation is compatible with the the solvable filtration of link concordance from \cite{COT03}.  To be precise, in Proposition \ref{prop:compatible} we show that a link is in $\F_n$ if and only if it is $n$-solvably concordant to the unlink.  While we do not explore the relationship in this paper, $n$-solvable concordance closely related to $n$-solvable cobordism of link exteriors, as in \cite{Cha2014}.  In Section \ref{sect:proof} we prove  the second main result of this paper, Theorem \ref{thm:mainLink}.  The notation $\C^\mu$ denotes concordance of $\mu$-component links, $\widehat \C^\mu$ denotes homology concordance of $\mu$-component links in homology spheres, and $\simeq_n$ denotes $n$-solvable concordance.
   
   \begin{theorem}\label{thm:mainLink}
   Let $L$ be a link in a homology sphere.  Then for every $n\in \Z_{\ge0}$ there exists a link $L'$ in $S^3$ such that $(M,L)$ is  $n$-solvably concordant to $(S^3,L')$ .  In other words, $\Psi:\C^\mu/\simeq_n \to \widehat \C^\mu/\simeq_n$ is bijective.
   \end{theorem}
   
   Restricting to the setting of knots, $\C/\F_n = \C/\simeq_n$ and $\widehat\C/\widehat\F_n = \widehat \C/\simeq_n$, as we observe in Proposition \ref{prop:compatible} and Remark \ref{rem:compatible}.
    Thus, Theorem \ref{thm:mainKnot} is an immediate consequence of Theorem \ref{thm:mainLink}.

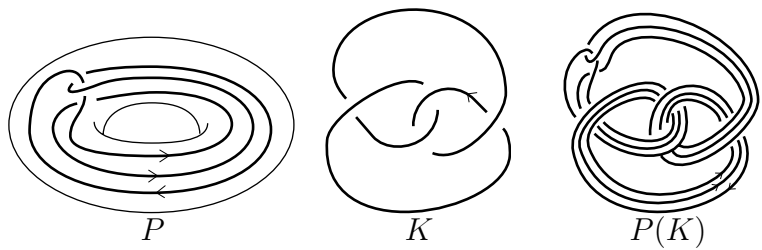
\begin{figure}[h!]
\begin{center}
\begin{picture}(310,90)
\put(0,10){\includegraphics{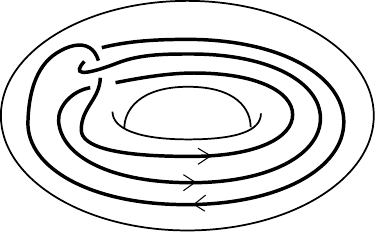}}
\put(50,0){$P$}
\put(120,10){\includegraphics{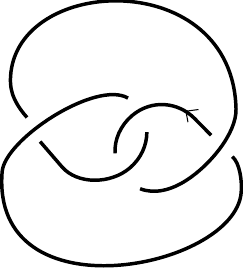}}
\put(150,0){$K$}
\put(210,10){\includegraphics{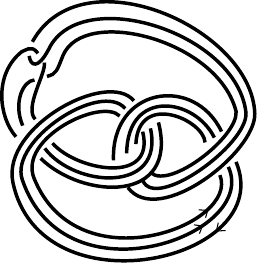}}
\put(235,-0.1){$P(K)$}
\end{picture}
\end{center}
\caption{The satellite construction on knots in $S^3$.  Left to right:  A pattern, a companion knot, the resulting satellite knot}\label{fig:satelliteoperation}
\end{figure}

   A \emph{pattern} $P$ consists of a knot in a solid torus.  Any pattern gives a map on knot concordance, $K\mapsto P(K)$ via the \emph{satellite construction}.  See Figure \ref{fig:satelliteoperation}.   An easy argument \cite[Proposition 3.1]{DaRa2017} reveals that no pattern of winding number different from $\pm1$ can produce a surjective map on $\C$.   In \cite{DaRa2017} Ray and the author define a group $\widehat\S$ which acts on $\widehat \C$ together with a map $E$ from the set of winding number $\pm1$ satellite operators to $\widehat\S$ making the following diagram commute whenever $P$ is a winding number $\pm1$ pattern:
   $$
   \begin{tikzcd}
\C_\Z \arrow{r}{P} \arrow[hook]{d}{\Psi}
& \C_\Z \arrow[hook]{d}{\Psi} \\
\widehat{\C} \arrow[hook, two heads]{r}{E(P)}
& \widehat{\C}.
   \end{tikzcd}
   $$ 
   Here $\C_\Z = \C/\ker(\Psi)$ is the integral knot concordance group.  Since $E(P):\widehat\C\to \widehat\C$ comes from a group action it is bijective. As done in \cite{DaRa2017} it follows from a straightforward diagram chase that $P:\C_\Z\to \C_\Z$ is injective.  A different proof of injectivity appears  in a previous work of Cochran, Ray, and the author \cite{CDR14}.    If $\Psi$ were surjective, then the same diagram chase would give that $P:\C_\Z\to \C_\Z$ is bijective.  This is noteworthy because as a consequence of \cite{Levine2016} there exists a winding number $1$ pattern $P$ for which $K\mapsto P(K)$ is not surjective on smooth integral knot concordance.  In Section \ref{sect:app} we recall the notions of \cite{DaRa2017}, verify that these notions are compatible with the solvable filtration, and use that $\Psi:\C/\F_n\to \widehat \C/\widehat\F_n$ is a bijection by Theorem \ref{thm:mainKnot} to prove the following.
   
   \begin{theorem}\label{thm:application}
   Let $P$ be a winding number $\pm1$ pattern.  Then for every $n\in \Z_{\ge 0}$ the satellite operator $P:\C/\F_n\to \C/\F_n$ is bijective.  
   \end{theorem}
   
%
%
%
%
%

   \subsection{Outline of the paper}
   
   In Section \ref{sect:defn} we formally state the definition of homology concordance, the solvable filtration, and solvable concordance of links in homology spheres.  We go on to prove some basic properties of this filtration.  In Section \ref{sect:proof} we give two entirely 3-dimensional results about handlebodies in homology spheres.  We close Section \ref{sect:proof} with the proof of Theorem \ref{thm:mainLink}. In Section \ref{sect:app} we combine Theorem \ref{thm:mainKnot} with the generalized satellite operation of \cite{DaRa2017} in order to prove Theorem \ref{thm:application}.  
   
    A reader looking only to get the core ideas of this paper may read Propositions \ref{prop: in a ball}, \ref{prop: solvable surgery}, \ref{prop:bigger hbody}, and \ref{prop:0-surg for hbody}.  The reader will then be able to follow the proof of Theorem \ref{thm:mainLink} at the end of Section \ref{sect:proof}. 
   
   \subsection{Acknowledgements}
   
   We would like to thank Jae Choon Cha, Daniel Kasprowski, Matthias Nagel, Carolyn Otto, Mark Powell,  Arunima Ray, and Peter Teichner   for helpful conversations.
   \section{homology concordance and the solvable filtration}\label{sect:defn}
   
   In this section we state explicitly the notion of homology concordance and the solvable filtration of knots and links in homology spheres.  Throughout this paper all manifolds are oriented and compact, or are covers of oriented compact manifolds.  All submanifolds are properly embedded and locally flat.  Given a locally flat properly embedded submanifold $F$ in the manifold  $W$, $\nu(F)$ refers to an open tubular neighborhood of $F$ and $E(F) = W-\nu(F)$ is the exterior of $F$.  
   
 For an oriented manifold $N$, $\overline{N}$ denotes the orientation reverse of $N$.   Let $M$ and $N$ be homology 3-spheres.  A 4-manifold $W$ bounded by $M\sqcup \overline{N}$ is called a \emph{homology cobordism} between $M$ and $N$ if the inclusion induced maps $H_*(M)\to H_*(W)\ot H_*(N)$ are isomorphisms.  Let $L=L_1\cup\dots\cup L_\mu$ and $J=J_1\cup\dots\cup J_\mu$ be $\mu$-component links in the homology spheres $M$ and $N$.  We say that $(M,L)$ is  \emph{homology concordant} to $(N,J)$ if there is a homology cobordism $W$ from $M$ to $N$ in which there exist $\mu$ disjoint locally flat properly embedded annuli $C = C_1\cup \dots\cup C_\mu$ with $\bdry C_i$ equal to $L_i\cup r(J_i)$, where $r(J_i)\subseteq \overline N$ is the reverse of $J_i$.  We call $(W,C)$ a \emph{homology concordance} from $(M,L)$ to $(N,J)$.  
 
Just as $\mathcal\C^\mu$ denotes concordance of $\mu$-component links in $S^3$, we denote by $\widehat{\C}^\mu$ the set of homology concordance classes of $\mu$-component links in homology spheres.  When $\mu=1$ we remove it from our notation, so that $\widehat \C$ denotes homology concordance of knots in homology spheres.    
 
    In \cite{COT03}, Cochran-Orr-Teichner introduce a filtration of $\C^\mu$.    Before we can state the definition we need some background.  For any group $G$, the derived series of $G$ is defined recursively by $G^{(0)}=G$, and $G^{(n+1)} = [G^{(n)}, G^{(n)}]$.  If $W$ is a 4-manifold and $G=\pi_1(W)$, then there is an equivariant intersection form 
    $$
    \lambda^W_n:H_2(W;\Z[G/G^{(n)}])\times H_2(W;\Z[G/G^{(n)}])\to \Z[G/G^{(n)}].
    $$
    When $W$ is understood, we leave it out of the notation, just saying $\lambda_n$.  With these notions in hand we can state the definition of the solvable filtration.  
   
   \begin{definition}[Section 8 of \cite{COT03} when $M=S^3$. See also Definition 2.1 of \cite{CocTei04}]\label{defn:solvable}
   Let $L$ be a $\mu$-component link in a homology sphere $M$.  Let $M_L$ denote its zero framed surgery.  We say that $(M,L)$ is \emph{$n$-solvable}, and $(M,L)\in \widehat\F_n^\mu$ if there exists a spin compact 4-manifold  $W$ called an \emph{$n$-solution} bounded by $M_L$ such that:
   \begin{enumerate}
   \item $H_1(W)\cong \Z^\mu$ is generated by the meridians of $L$.
   \item \label{solvable classes} There exist classes $x_1, \dots, x_k, y_1, \dots, y_k\in H_2(W;\Z[G/G^{(n)}])$ with $\lambda_n(x_i, y_j) = \delta_{i,j}$ (the Kronecker delta) and $\lambda_n(x_i,x_j)=\lambda_n(y_i,y_j)=0$.  Here $G=\pi_1(W)$.
   \item \label{proj basis}The image of $\{x_i, y_i\}$ in the projection $H_2(W;\Z[G/G^{(n)}])\to H_2(W)$ gives a basis for $H_2(W)\cong \Z^{2k}$.
   \end{enumerate}
  The classes $x_i$ and $y_i$ are called \emph{$n$-Lagrangians} and \emph{$n$-duals} respectively.  
   \end{definition}

   \begin{remark}\label{remark:compatible}  
   Notice that the only difference between the definition presented here and the definition in \cite{COT03} is that the ambient 3-manifold is allowed to be any homology sphere.  Thus, given a link $L$ in $S^3$, $L$ is $n$-solvable (as in \cite{COT03}) if and only if $(S^3, L)\in \widehat\F_n^\mu$.  As a consequence we recover $\F_n^\mu$ as the set of all $\mu$-component links in $S^3$ for which $(S^3,L)\in \widehat\F_n^\mu$.
   \end{remark}
   
Just as with $\C$ and $\widehat \C$, when $\mu=1$ we leave it out of the notation setting $ \F_n =  \F_n^1$ and $ \widehat\F_n =  \widehat\F_n^1$.  Similarly to the group structure on $\C$, $\widehat{\C}$ forms an abelian group under connected sum of pairs:  $(M,K)\#(N,J) = (M\# N, K\# J)$.  The identity element is given by the equivalence class of the unknot in $S^3$ and the inverse of $(M,K)$ is given by $-(M,K) = \left(\overline{M}, r(K)\right)$.  Since $\widehat \F_n\le \widehat \C$, one arrives at an equivalence relation on  $\widehat\C$ by studying the quotient group $\widehat\C/\widehat\F_n$.  In order to apply our techniques to the setting of links we need an analogous equivalence relation on $\widehat \C^\mu$ for which the equivalence class of the unknot is $\widehat \F^\mu_n$.    one could get such an equivalence relation by  studing the solvable filtration of the string link concordance group as in \cite{CHL4} or by considering the exteriors of links as bordered 3-manifolds up to $n$-solvable cobordism as in \cite{Cha2014}.  We prefer to study an equivalent notion which we call $n$-solvable concordance.  

     \begin{definition}\label{defn:n-solvableEquiv}
Let $L$ and $J$ be $\mu$-component links in the homology spheres $M$ and $N$.  We say that $(M,L)$ is \emph{$n$-solvably concordant} to $(N,J)$ if there exists a spin cobordism $W$ from $M$ to $N$ such that:
\begin{enumerate}
\item $H_1(W)=0$, so that $W$ is an $H_1$-cobordism.
\item There exist disjoint locally flat properly embedded annuli $C = C_1\cup\dots\cup C_\mu$ with $\bdry C_i$ equal to $L_i\cup r(J_i)$.  
\item\label{equiv classes} 
There exist classes $x_1, \dots, x_k, y_1, \dots, y_k\in H_2(E(C);\Z[G/G^{(n)}])$ with $\lambda_n(x_i, y_j) = \delta_{i,j}$ and $\lambda_n(x_i,x_j)=\lambda_n(y_i,y_j)=0$.  Here $G=\pi_1(E(C))$.
\item The image of $\{x_i, y_i\}$ in the composition $H_2(E(C);\Z[G/G^{(n)}])\to H_2(E(C))\to H_2(W)$ gives a basis for $H_2(W)\cong \Z^{2k}$.  
\end{enumerate}
We use $\simeq_n$ to denote this equivalence relation.  The pair $(W,C)$ is called an \emph{$n$-solvable concordance}.  
\end{definition}


   The compatibility of $n$-solvable concordance with the solvable filtration of \cite{COT03} comes as no surprise.  Indeed the proof of the following proposition amounts to a direct check of hypotheses.  

\begin{proposition}\label{prop:compatible}
Let $(M,L)$ be a link in a homology sphere.  Then $(M,L)$ is $n$-solvably concordant to the unlink in $S^3$ if and only if $(M,L)\in \widehat \F_n^\mu$.
\end{proposition}

\begin{remark}\label{rem:compatible}
Both $\widehat \C/\simeq_n$ and $\widehat \C/\widehat \F_n$ give group quotients of $\widehat \C$.   Proposition \ref{prop:compatible} shows that the kernels of the maps from $\widehat\C$ to these quotients agree, and so these quotients are equal.  
\end{remark}
\begin{proof}

  Suppose that $(M,L)$ is $n$-solvable and let $W_0$ be an $n$-solution.  Now, $\bdry W_0 = M_L$ is the $0$-surgery on $L$.   Let $W_1$ be given by by adding to $W_0$ a 2-handle to the $0$-framing on a meridian of each component of $L$.  As these meridians form a basis for $H_1(W_0)$, it follows that $H_1(W_1)=0$, and $H_2(W_0)\to H_2(W_1)$ is an isomorphism.  These meridians give helper circles which cancel with the 0-surgery on the components of $L$, so that $\bdry W_1 = M$.  In $W_1$ the components of $L$ bound the co-cores of the added 2-handles.  Call these co-cores $\Delta_1,\dots, \Delta_\mu$ and let $\Delta = \Delta_1\cup\dots\cup \Delta_\mu$.  Let $p_i$ be a point interior to $\Delta_i$ and $B_i\subseteq \nu(\Delta_i)\subseteq W_1$ be a small open 4-ball intersecting $\Delta_i$ in a disk containing $p_i$.  Take $\alpha_i$ to be an arc running from $B_i$ to $B_{i+1}$ disjoint from $\Delta$ and let $B$ be the result of tubing the various $B_i$ together along the arcs $\alpha_i$ to get a single 4-ball.  Let $W=W_1-B$, $C_i=\Delta_i- B$ and $C=C_1\cup\dots\cup C_\mu$.  As the removal of a 4-ball interior to $W_1$ does not change first homology, $H_1(W)\cong H_1(W_1)=0$ so that $W$ is an $H_1$-cobordism from $M$ to $S^3$.  The locally flat properly embedded annulus $C_i$ is bounded by $L_i$ and the corresponding component of the unlink in $S^3$.  It remains to check that $(W,C)$ is an $n$-solvable concordance.  
  

  Notice that  $E(C)$ could be constructed from $W_1$ by first cutting out a neighborhood of $\Delta$ and then cutting out neighborhoods of the arcs $\alpha_i$.  Removing a neighborhood of $\Delta$ from $W_1$ recovers the $n$-solution $W_0$.  Thus, $E(C)$ is the result removing neighborhoods of the properly embedded arcs $\alpha_1,\dots,\alpha_{\mu-1}$ from $W_0$. In other words, $W_0$ is $E(C)$  together with $(\mu-1)$ 3-handles.  Since addition of 3-handles does not change fundamental group, the inclusion induced map $\iota_*:\pi_1(E(C))\to \pi_1(W_0)$ is an isomorphism.   As $\iota_*$ gives a preferred isomorphism between these two groups we call each of them $G$.  Again since these spaces are related by 3-handle addition  $\iota_*:H_2(E(C))\onto H_2(W_0)$ is a surjection with coefficients in either $\Z$ or $\Z[G/G^{(n)}]$.

   Let $x_1, y_1,\dots x_k, y_k$ in $H_2(W_0; \Z[G/G^{(n)}])$ be the $n$-Lagrangians and $n$-duals assumed to exist in Definition \ref{defn:solvable} since $W$ is an $n$-solution.   Let $x_i'$ and $y_i'$ be a choice of preimages of $x_i$ and $y_i$ under $\iota_*:H_2(E(C);\Z[G/G^{(n)}])\onto H_2(E(C);\Z[G/G^{(n)}])$.  The functoriality of the intersection form implies that $\lambda_n^{E(C)}(x_i', y_i') = \lambda_n^{W_0}(x_i, y_i) = \delta_{i,j}$.  Similarly, $\lambda_n^{E(C)}(x_i', x_i') =\lambda_n^{E(C)}(y_i', y_i')=0$ so that Condition \pref{equiv classes} of Definition \ref{defn:n-solvableEquiv} is satisfied.  By Condition \pref{proj basis} of Definition \ref{defn:solvable}, the projection of $\{x_i, y_i\}$ to $H_2(W_0)$ is a basis.   Consider the following commutative diagram whose horizontal maps are induced by inclusion and whose vertical maps are induced by coverings
   $$\begin{tikzcd}
    H_2(E(C); \Z[G/G^{(n)}]) \arrow[ two heads]{rrr}\arrow{d} &&& H_2(W_0; \Z[G/G^{(n)}])\arrow{d}\\    
    H_2(E(C))  \arrow{r} &H_2(W)\arrow["\cong"]{r} & H_2(W_1)& H_2(W_0).\arrow["\cong"]{l}    
\end{tikzcd}
   $$
   Recall that $H_2(W_0)\cong H_2(W_1)$ since $W_1$ is constructed by attaching 2-handles to a basis for $H_1(W_0)$.   We built $W$ from $W_1$ by removing an interior 4-ball so $H_2(W)\cong H_2(W_1)$.  Since the image of $\{x_i, y_i\}$ in $H_2(W_0;\Z[G/G^{(n)}])\to H_2(W_0)$ gives a basis, the diagram above implies that the image of $\{x_i', y_i'\}$ in the composition $H_2(E(C); \Z[G/G^{(n)}]) \to H_2(E(C)) \to H_2(W)$ gives a basis for $H_2(W)$, completing the proof that $(W,C)$ is an $n$-solvable concordance and $(M,L)$ is $n$-solvably concordant to the unlink.
   
   The proof of the reverse implication is basically the same.  Start with an $n$-solvable concordance $(W_0,C)$ from $(M,L)$ to the unlink in $S^3$.   Cap the $S^3$-boundary component of $W_0$ with a 4-ball to get $W_1$.  Cap the unlink boundary components of $C$ with disks to get a collection of embedded disks $\Delta\subseteq W_1$ bounded by $L$.  A reasonably direct check of hypotheses reveals that the exterior of $\Delta$ in $W_1$ is an $n$-solution for $(M,L)$.  
\end{proof}   
 
      
The remainder of this section is devoted to the proof of a pair of propositions regarding $\widehat \C^\mu/\simeq_n$.     

   \begin{proposition}\label{prop: in a ball}
   Suppose that $(M,L)$ is a $\mu$-component link in a homology sphere and that there exists a closed 3-ball $B\subseteq M$ with $L\subseteq B$.  Then $(M,L)$ is in the image of $\Psi:\C^\mu\to \widehat\C^\mu$ and so is in the image of the induced map $\Psi:\C^\mu/\simeq_n\to \widehat\C^\mu/\simeq_n$.
   \end{proposition}
\begin{proof}
Suppose that $(M,L)$ is a link in a homology sphere and that there exists a closed 3-ball $B$ with $L\subseteq B\subseteq M$.  Then $M-\int B$ is a homology 3-ball and by Freedman-Quinn \cite[Corollary 9.3C]{FQ} $M-\int B$ is  homology cobordant to the 3-ball $B^3$.    By gluing $B\times[0,1]$ to a homology cobordism from $M-\int B$  to $B^3$, one obtains a homology cobordism $W$ from $M$ to $S^3$ in which the image of $L\times[0,1]\subseteq B\times[0,1] \subseteq W$ is a concordance between $L$ and a link in $S^3$.  Thus, $(M,L)$ is homology concordant to some link in $S^3$ and so $(M,L)$ is in the image of $\Psi:\C^\mu\to \widehat \C^\mu$.  
\end{proof}

The second property we need is a surgery which preserves solvable concordance.  The conditions on the curves we use to perform surgery are slightly long and we reference them often, so we make the following definition.

\begin{definition}\label{defn: solvable surgery}

   Suppose that $(M,L)$ is a link in a homology sphere.  Let $\{\alpha_i, \beta_i~|~i=1,\dots, k\}$ be a collection of embedded simple closed curves in the exterior of $L$ such that:
   \begin{enumerate}
   \item There are surfaces $A_1, B_1,\dots, A_k,B_k$ embedded in the exterior of $L$ with $\bdry A_i = \alpha_i$ and $\bdry B_i=\beta_i$.
   \item $A_i$ intersects $\beta_i$ transversely in a single point, $B_i$ intersects $\alpha_i$ transversely in a single point, and for all $i\neq j$ $A_i\cap \alpha_j = B_i\cap \alpha_j = A_i\cap \beta_j = B_i\cap \beta_j = \emptyset$.
   \item $\pi_1(A_i)\subseteq \pi_1(M-L)^{(n)}$ and $\pi_1(B_i)\subseteq \pi_1(M-L)^{(n)}$.
   \end{enumerate}
   This collection of curves is called a collection of $n$-\emph{solvable surgery curves} for $L$.  We denote by $M_{\alpha_1,\dots, \alpha_k, \beta_1,\dots, \beta_k}$ the result of performing $0$-surgery along these curves, and by $L_{\alpha_1,\dots, \alpha_k, \beta_1,\dots, \beta_k}$ the image of $L$ in this 3-manifold.  We will often abuse notation  and instead say $M_{\alpha, \beta}$ and $L_{\alpha, \beta}$.  
\end{definition}

   \begin{proposition}\label{prop: solvable surgery}
   Suppose that $(M,L)$ is a link in a homology sphere.  Let $\{\alpha_i, \beta_i~|~i=1,\dots, k\}$ be a collection of $n$-{solvable surgery curves} for $L$.  Then $M_{\alpha, \beta}$ is a homology sphere and $(M,L)$ is $n$-solvably concordant to $(M_{\alpha, \beta},L_{\alpha, \beta})$.  
   \end{proposition}
   \begin{proof}
   
   We first justify that $M_{\alpha, \beta}$ is a homology sphere.  This follows since $H_1(M_{\alpha, \beta})$ is presented by the linking-framing matrix for the surgery curves $\alpha_1, \beta_1, \dots, \alpha_k, \beta_k$.  The surfaces $A_i$ and $B_i$ can be used to see that the linking-framing matrix is 
   $\underset{k}{\bigoplus}\left(\begin{array}{cc}0&\pm1\\\pm1&0\end{array}\right)$.  As this matrix presents the trivial group, $H_1(M_{\alpha,\beta})=0$ and $M_{\alpha,\beta}$ is a homology sphere.
   
   Next we construct an $n$-solvable concordance.  Start with $M\times[0,1]$ and add 2-handles to the $0$-framings on $\alpha_i\times\{1\}$ and $\beta_i\times\{1\}$. Call the resulting manifold $W$.  Let $C = C_1\cup\dots\cup C_\mu$ where $C_i$ is the image of $L_i\times[0,1]$ in $W$.  Notice that $\bdry W = M_{\alpha, \beta}\cup \overline M$, and that $\bdry C_i = (L_i)_{\alpha,\beta}\cup r(L_i)$.    We claim that $(W, C)$ is an $n$-solvable concordance.

As we only added 2-handles  along nullhomologous curves, $H_1(W)\cong H_1(M\times[0,1])=0$, so that  $W$ is an $H_1$-cobordism.  Let $A_1,B_1,\dots, A_k, B_k\subseteq M$ be the surfaces assumed to exist since $\alpha, \beta$ are $n$-solvable surgery curves.  
Pick $2k$ distinct numbers $0<\epsilon_1<\delta_1<\dots<\epsilon_k<\delta_k<1$.  Consider the pushed in surfaces $A_i' = A_i\times\{\epsilon_i\}\cup \alpha_i\times[\epsilon_i,1]\subseteq M\times[0,1]$ and $B_i' = B_i\times\{\delta_i\}\cup \beta_i\times[\delta_i,1]\subseteq M\times[0,1]$.  As these surfaces have all been pushed a different distance into $M\times[0,1]$, they are disjoint except that $B_i'$ intersects  $A_i'$ transversely in a single point since $B_i$ intersects $\alpha_i$ in a single point.  
Add to $A_i'$ the core of the 2-handle added along $\alpha_i$ to get a closed surface $X_i\subseteq W$.  Do the same to $B_i'$ to get $Y_i\subseteq W$.   Since these closed oriented embedded surfaces form intersection duals to the co-cores of the added 2-handles, $\{X_1,  Y_1, \dots, X_k, Y_k\}$ forms a basis for $H_2(W)$.  Since the trivialization of the normal bundles of $A_i'$ and of the core of the 2-handle added to $a_i$ both induce the $0$-framing on $\alpha_i$, $X_i$ has a trivial normal bundle.  In particular $X_i$ has a pushoff in $W$ which is disjoint from $X_i$.  The same follows for $Y_i$.    

The diagram below commutes.  The left pointing isomorphisms are induced by the projection $E(K)\times[0,1]\to E(K)$ and all others are induced by inclusion.
$$
\begin{tikzcd}
\pi_1(A_i)\arrow{d}&\pi_1(A_i') \arrow[r,twoheadrightarrow ] \arrow[d] \arrow["\cong"]{l}
& \pi_1(X_i) \arrow[d] 
\\
\pi_1(E(K))&\pi_1(E(K)\times[0,1]) \arrow[r] \arrow["\cong"]{l}
& \pi_1(E(C)).
\end{tikzcd}
$$
By assumption $\pi_1(A_i)\subseteq \pi_1(E(K))^{(n)}$ so that $\pi_1(X_i)\subseteq \pi_1(E(C))^{(n)}$.  Thus, $X_i$ lifts to an embedded surface $\widetilde X_i$ in  $\widetilde {E(C)}_n$, the cover corresponding to $\pi_1(E(C))^{(n)}$.  Since the $X_i$ are all disjoint and each is disjoint from its own pushoff, it follows that for all $i,j=1,\dots, k$ and every $\gamma$ in the deck group of $\widetilde {E(A)}_n$, $\widetilde X_i$ is disjoint from a pushoff of $\gamma\left(\widetilde X_j\right)$.  The same argument reveals that $\widetilde Y_i$ is disjoint from a pushoff of $\gamma \left(\widetilde Y_j\right)$, where $\widetilde Y_i$ is a lift of $Y_i$, and that $\widetilde X_i$ is disjoint from $\gamma\left(\widetilde Y_j\right)$ as long as $i\neq j$.  Since $X_i$ intersects $Y_i$ in a single point these lifts can be chosen so that $\widetilde X_i$ intersects $\widetilde Y_i$ transversely in a single point and is disjoint from every $\gamma \left( \widetilde Y_i\right)$ with $\gamma$ a deck transform not equal the identity.  Let $x_i$ and $y_i$ be the classes of $\widetilde X_i$ and $\widetilde Y_i$ in $H_2(\widetilde {E(C)}_n) = H_2(E(C);\Z[\pi/\pi^{(n)}])$.  Since geometric intersection numbers recover algebraic, we have that $\lambda^{E(C)}_n(x_i, x_j) = \lambda^{E(C)}_n(y_i, y_j) =  0$ and $\lambda^{E(C)}_n(x_i, y_j) = \delta_{i,j}$.  

As $x_i, y_i$ project down to the homology classes of $X_i, Y_i$, and $\{X_i, Y_i\}$ forms a basis for $H_2(W)$, we see that $(W,C)$ is an $n$-solvable concordance from $(M,L)$ to $(M_{\alpha, \beta},L_{\alpha, \beta})$.  This completes the proof.
\end{proof}

\section{The Proof of Theorem \ref{thm:mainLink}}\label{sect:proof}

In this section we prove that every link in a homology sphere is $n$-solvably concordant to some link in $S^3$.  The proof follows quickly from Propositions \ref{prop: in a ball} and \ref{prop: solvable surgery} from the previous section and two purely 3-dimensional results involving embedded handlebodies in homology 3-spheres.  The first is essentially a special case of a result of Smythe \cite[Corollary]{Smythe70}.   

\begin{proposition}\label{prop:bigger hbody}
Let $M$ be a 3-dimensional homology sphere and $U\subseteq M$ be a handlebody embedded in $M$.  Then there exists a handlebody $V\subseteq M$ containing $U$ for which the inclusion induced map $H_1(M-V)\to H_1(M-U)$ is the zero homomorphism.
\end{proposition} 
\begin{proof}

In \cite[Corollary]{Smythe70}, Smythe proves that if $U$ is a handlebody embedded in a 3-manifold $M$ and $H_1(U)\to H_1(M)$ is the $0$-homomorphism then there exists a handlebody $V$ with $U\subseteq V\subseteq M$ such that $H_1(U)\to H_1(V)$ is the $0$-homomorphism.  Since $H_1(M)=0$, Smythe's result applies, and we see a handlebody V with $U\subseteq V\subseteq M$ and $H_1(U)\to H_1(V)$ is the $0$-homomorphism.  Since $M$ is a homology sphere, Alexander duality 
 implies that $H^1(M-U)\to H^1(M-V)$ is the $0$-homomorphism.  An application of the universal coefficient theorem completes the proof.

\end{proof}

The second 3-dimensional result we shall need gives a means of performing surgery to ``unlink'' an embedded handlebody in a homology sphere from every curve in a surgery presentation for that homology sphere so that the resulting handlebody lies in a 3-ball in the image of surgery.  

\begin{proposition}\label{prop:0-surg for hbody}
Let $M$ be a homology sphere and $U\subseteq M$ be a handlebody embedded in $M$.  Then there exist curves $\alpha_1,\beta_1\dots, \alpha_k,\beta_k$ in $M-U$ such that:
\begin{enumerate}
\item  $\lk(\alpha_i, \alpha_j) =\lk(\alpha_i, \beta_j) = \lk(\beta_i, \beta_j) = 0$ for all $i\neq j$. 
\item $\lk(\alpha_i, \beta_i) = 1$ for all $i$.
\item 
There exists a 3-ball in the result of $0$-surgery $M_{\alpha, \beta}$ which contains the image of $U$.
\end{enumerate}
\end{proposition}

\begin{proof}

Let $M$ be a homology sphere and $U\subseteq M$ be a handlebody.  By picking a particular identification of $U$ with an abstract handlebody we may realize $U$ as a regular neighborhood of a wedge of circles $C = c_1\vee c_2\vee\dots\vee c_g$ where $g$ is the genus of $U$.

Realize $M$ as surgery along some framed link $\Gamma:=\gamma_1 \cup \dots\cup \gamma_n$ in $S^3$ and $C$ as a knotted wedge of circles in the exterior of $\Gamma$.  Since $M$ is a homology sphere, we may slide the various $c_i$ over the $\gamma_j$ until we arrive at a diagram of $C$ in $M$ for which $\lk(c_i, \gamma_j) = 0$ for all $i,j$.  These handle slides amount to an isotopy of $C$ in $M$.

Thus, we may pick a diagram for the knotted wedge of circles $C$ in the exterior of $\Gamma$ so that $\lk(c_i, \gamma_j)=0$ for all $i,j$.  This means that for every time $c_i$ crosses over $\gamma_j$, there is another crossing with the opposite sign.  Two such crossings appear in Figure \ref{fig:crossing change} (a).  As depicted  in Figure \ref{fig:crossing change} (b), we perform $0$-surgery on a pair of curves $\alpha$ and $\beta$ nullhomologous in the complement of $\Gamma$ with $\lk(\alpha, \beta)=1$.  In Figure \ref{fig:crossing change} (c) we slide $c_i$ over the $0$-framing of $\alpha$.  The curve $\beta$ is now a helper circle for $\alpha$ so we may cancel these surgery curves, as in Figure \ref{fig:crossing change} (e).  The resulting diagram for $C$ is the same as we started with, except that the two overcrossings we considered are now undercrossings.  
Iterate this procedure until no $c_i$ crosses over any $\gamma_j$.  We have now found a family of curves, $\alpha_1, \beta_1,\dots, \alpha_k, \beta_k$,  each of which is nullhomologous in the exterior of $\Gamma$ so that: 
\begin{enumerate}
\item For all $i\neq j$, $\lk(\alpha_i, \alpha_j) = \lk(\alpha_i, \beta_j) =  \lk(\beta_i, \beta_j) = 0$.
\item For all $i$, $\lk(\alpha_i, \beta_i) = 1$.

\item The result of performing $0$-surgery along these curves sends $C$ to a new wedge of circles, $C_{\alpha, \beta}$ so that every crossing between $C_{\alpha, \beta}$ and every curve in a surgery diagram for $M_{\alpha, \beta}$ is an undercrossing.
\end{enumerate}

\begin{figure}
\setlength{\unitlength}{1pt}
\begin{picture}(400,120)
\put(0,15){\includegraphics[width=.25\textwidth,angle=90]{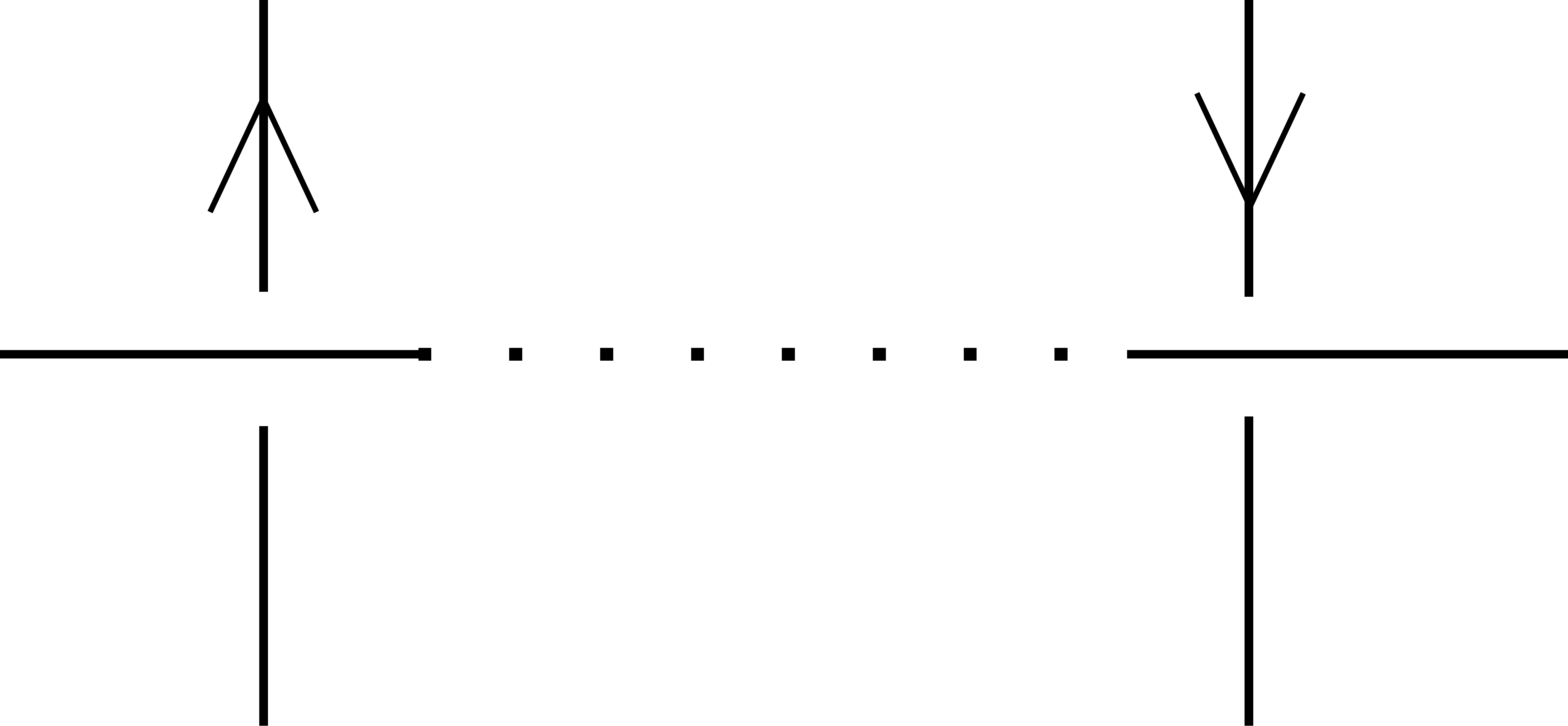}}
\put(15,17){$c_i$}
\put(40,40){$\gamma_j$}
\put(40,100){$\gamma_j$}
\put(25,0){(a)}

\put(70,15){\includegraphics[width=.25\textwidth,angle=90]{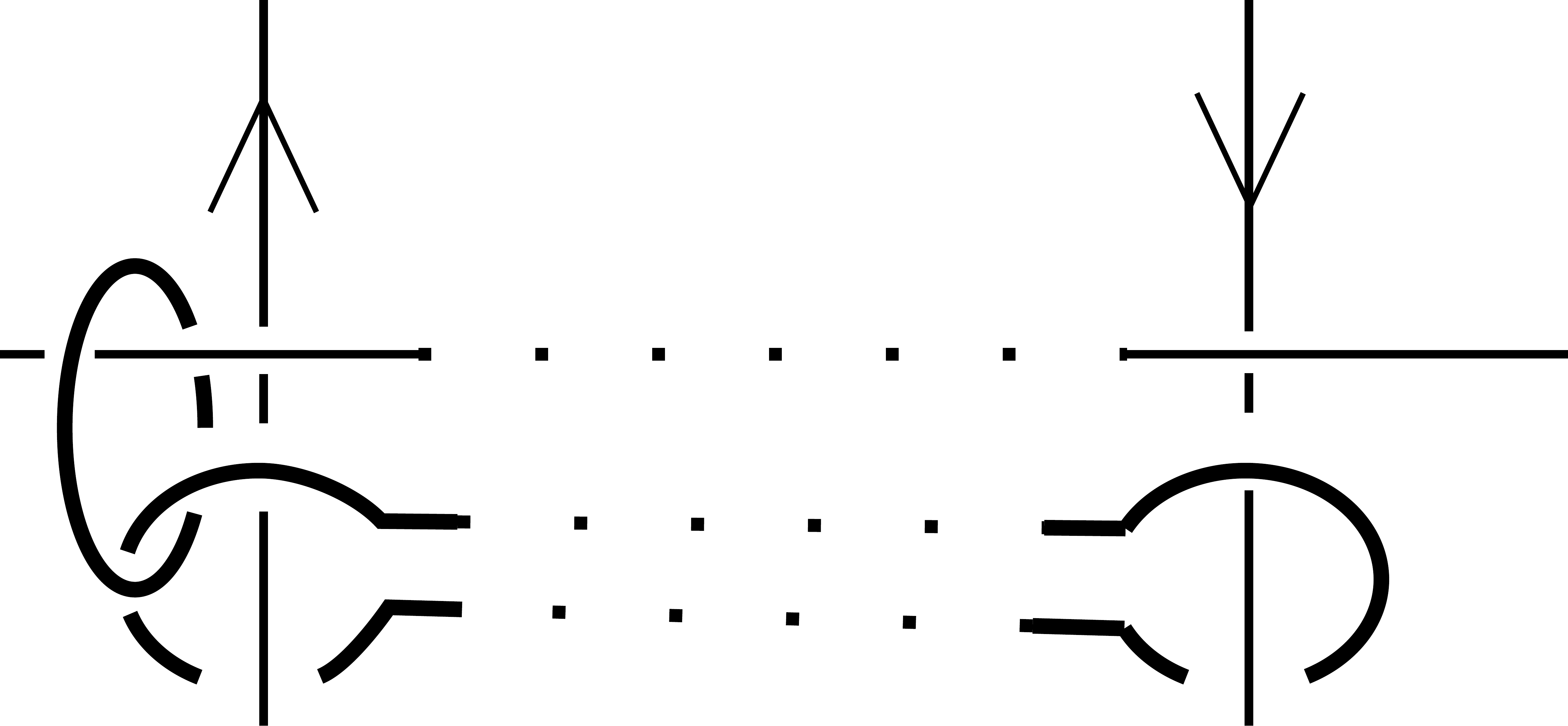}}
\put(117,45){$\alpha$}
\put(80,17){$\beta$}
\put(95,0){(b)}

\put(140,15){\includegraphics[width=.25\textwidth,angle=90]{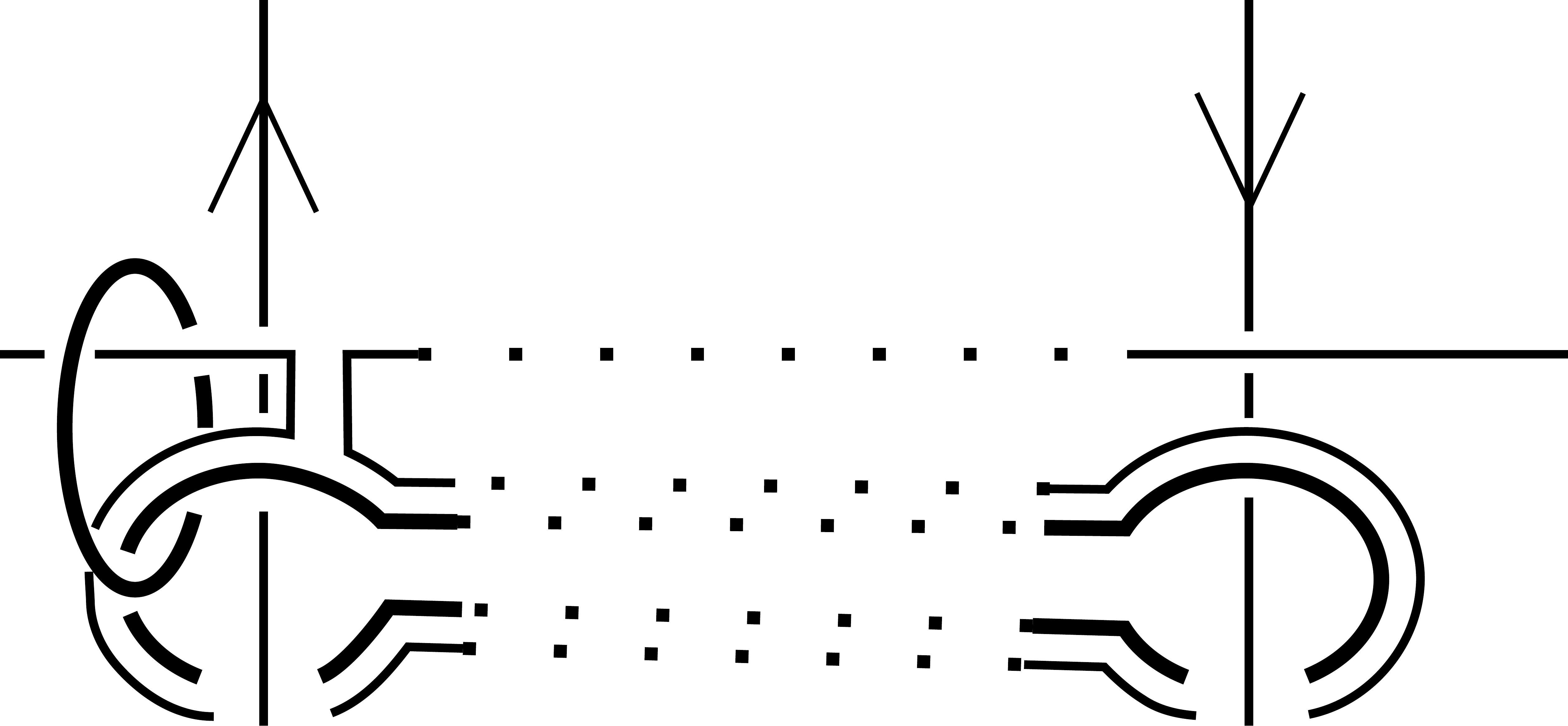}}
\put(165,0){(c)}

\put(210,15){\includegraphics[width=.25\textwidth,angle=90]{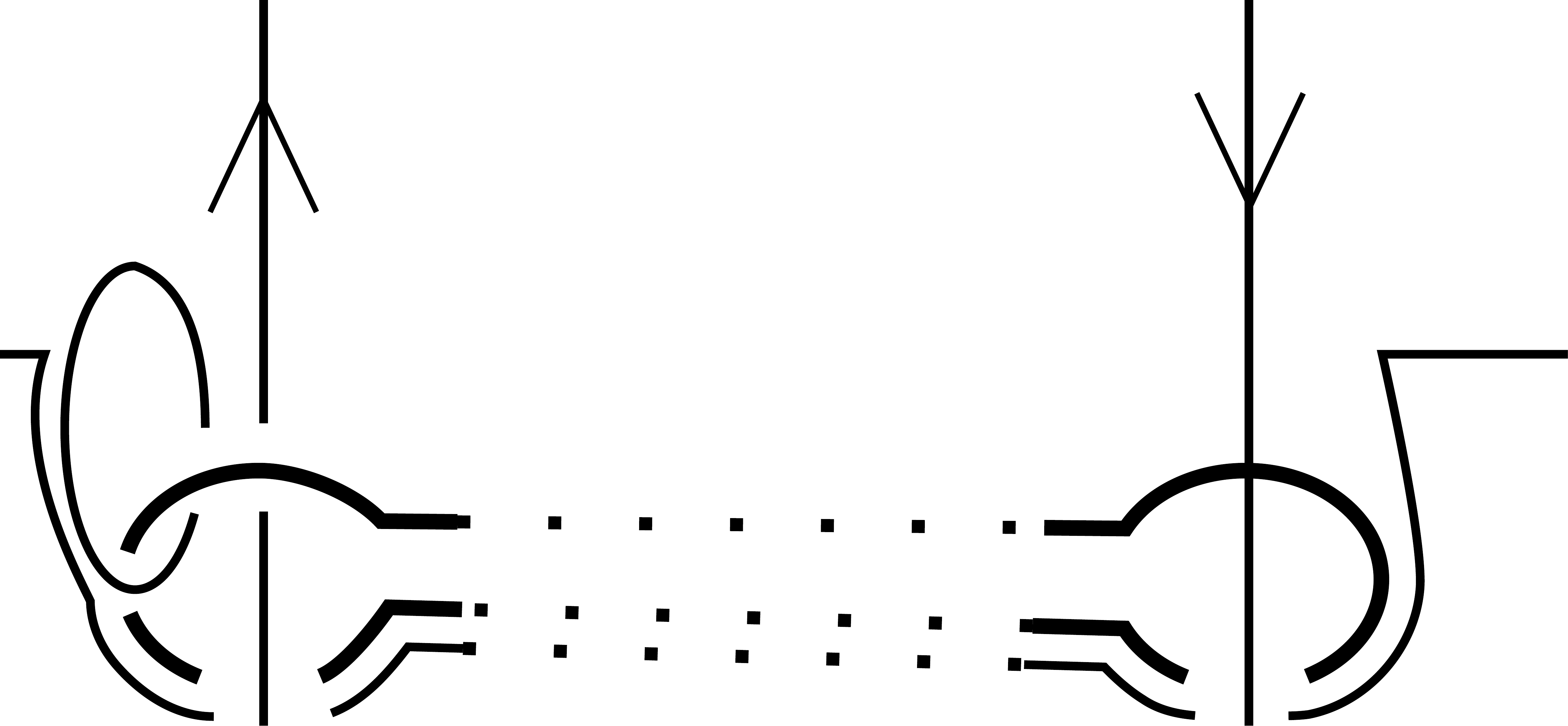}}
\put(235,0){(d)}

\put(280,15){\includegraphics[width=.25\textwidth,angle=90]{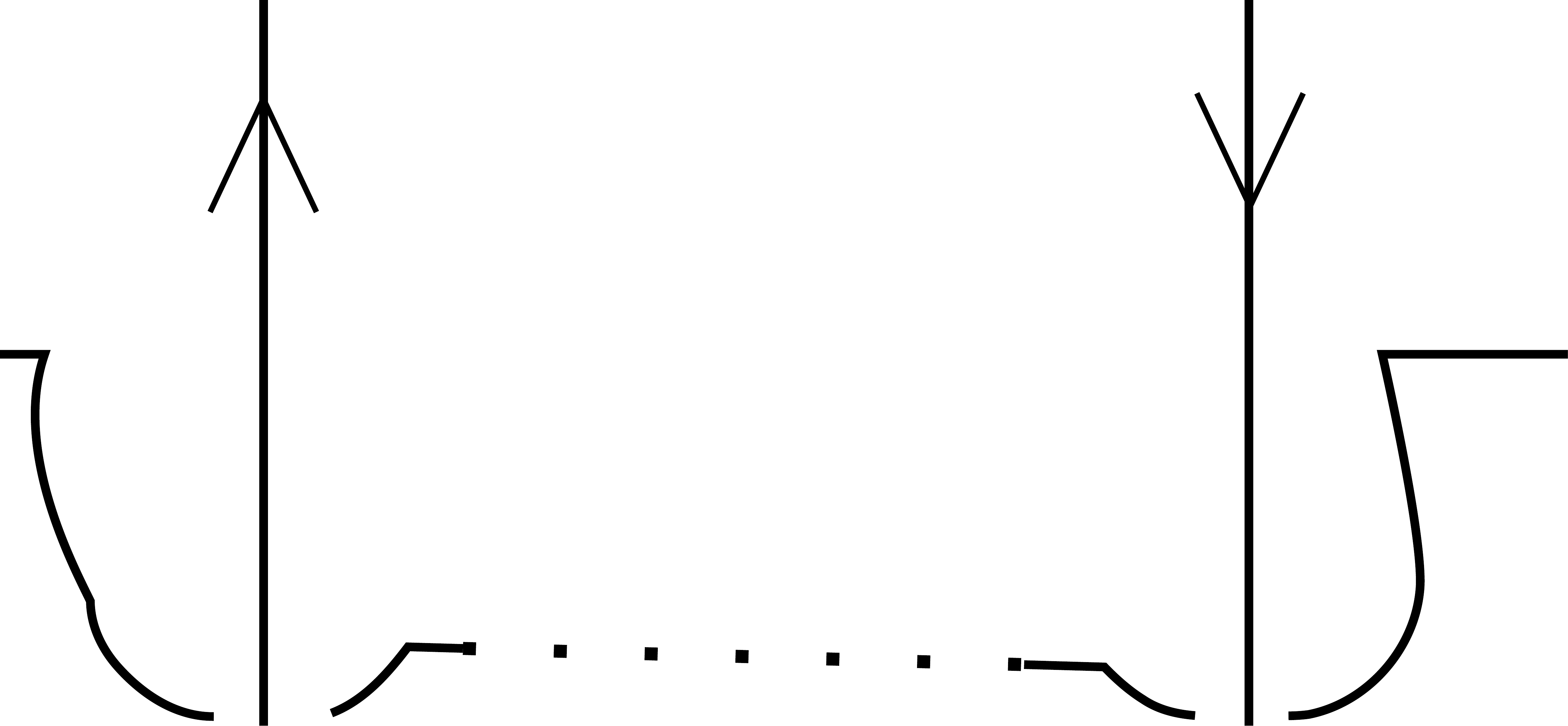}}
\put(305,0){(e)}

\put(350,15){\includegraphics[width=.25\textwidth,angle=90]{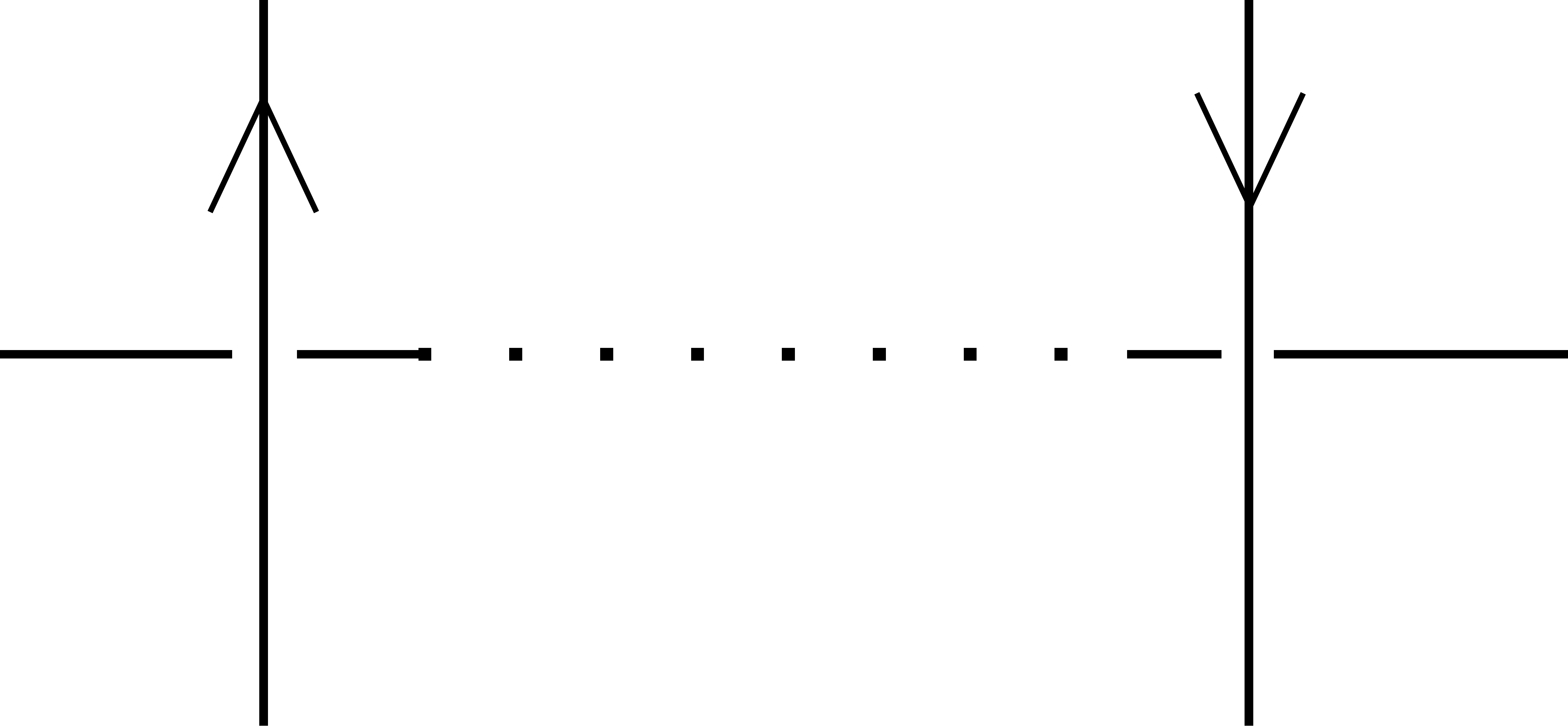}}
\put(375,0){(f)}
\end{picture}
\caption{Left to Right:  (a) A pair of crossings between $c_i$ and $\gamma_j$ with opposite signs.  (b) The result of performing $0$-framed surgery along a pair of curves, $\alpha$ and $\beta$.  (c) Sliding $c_i$ over the $0$-framing of $\alpha$. (d) After an isotopy, $\alpha$ is a helper circle for $\beta$. (e) Cancelling $\alpha$ and $\beta$. (f)  An isotopy produces the same diagram as (a) with two crossings changed.  }\label{fig:crossing change}
\end{figure}
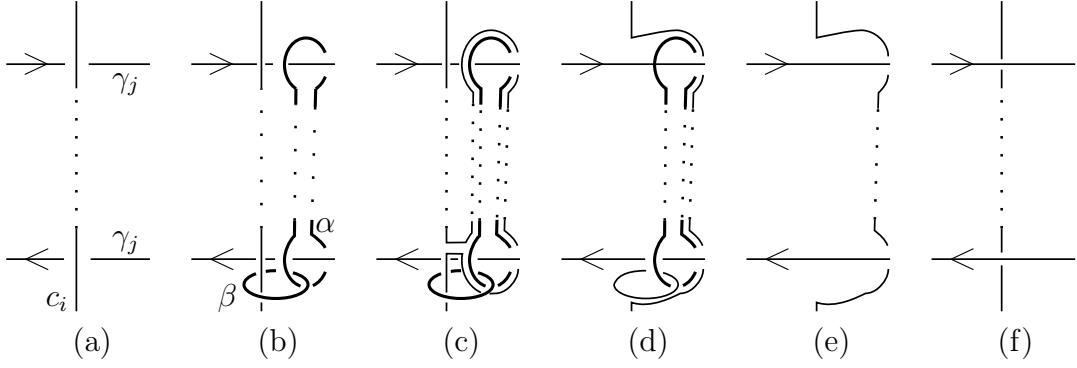

  We may further isotope $C_{\alpha, \beta}$ in the complement of $\Gamma$ until there are no crossings between $C_{\alpha, \beta}$ and $\Gamma$. Thus, $C_{\alpha, \beta}\subseteq M_{\alpha,\beta}$ is contained in a 3-ball.  Since $U_{\alpha, \beta}$,  the image of $U$ in this surgery,  is a regular neighborhood of $C_{\alpha, \beta}$, this completes the proof the proposition.

\end{proof}

Given Propositions \ref{prop: in a ball}, \ref{prop: solvable surgery}, \ref{prop:bigger hbody}, and \ref{prop:0-surg for hbody} we are ready to prove Theorem \ref{thm:mainKnot}.

\begin{proof}[Proof of Theorem \ref{thm:mainKnot}]
Let $(M,L)$ be a $\mu$-component link in a homology sphere and consider any any $n\in \Z_{\ge 0}$.  We must construct a link $L'$ in $S^3$ such that $(M,L)$ is $n$-solvably concordant to $(S^3, L')$.

 Let $U_0\subseteq M$ be the genus $\mu$ handlebody constructed by starting with tubular neighborhoods of the components of $L$ and tubing them together.  Since $U_0$ is a handlebody, we may iteratively apply Proposition \ref{prop:bigger hbody} to build a sequence of handlebodies $L\subseteq U_0\subseteq U_1\subseteq U_2\subseteq \dots\subseteq U_{n}\subseteq U_{n+1}\subseteq M$ such that for all $i$, the inclusion induced map, $H_1(M-U_{i+1}) \to H_1(M-U_i)$ is the $0$-homomorphism.  Thus, $\pi_1(M-U_{i+1})\subseteq  \pi_1(M-U_i)^{(1)}$.  The functoriality of the derived series implies that $\pi_1(M-U_n)\subseteq \pi_1(M-L)^{(n)}$.  
 
 By Proposition \ref{prop:0-surg for hbody} there exist curves $\alpha_1,\beta_1\dots, \alpha_k,\beta_k$ in $M-U_{n+1}$ such that
\begin{enumerate}
\item  $\lk(\alpha_i, \alpha_j) =\lk(\alpha_i, \beta_j) = \lk(\beta_i, \beta_j) = 0$ for all $i\neq j$ 
\item $\lk(\alpha_i, \beta_i) = 1$ for all $i$
\item There exists a 3-ball in $M_{\alpha, \beta}$ which contains the image of $U_{n+1}$.
\end{enumerate}

Let $L_{\alpha, \beta}$ be the image of $L$ in the result of this surgery.  Then $L_{\alpha,\beta}$ is contained in the image of $U_{n+1}$ in $M_{\alpha, \beta}$, which in turn is contained in a 3-ball.  Thus, Proposition \ref{prop: in a ball} implies that $L_{\alpha, \beta}$ is homology-concordant to some link $L'$ in $S^3$.  It remains only to verify that $\{\alpha_i, \beta_i:i=1,\dots, k\}$ is a collection of $n$-solvable surgery curves for $(M,L)$, as in Definition \ref{defn: solvable surgery}.   Proposition \ref{prop: solvable surgery} will then  conclude that $(M,L)$ is $n$-solvably concordant to $(M_{\alpha, \beta}, L_{\alpha, \beta})$.  

By assumption $\alpha_1, \beta_1,\dots, \alpha_k,\beta_k$ lie in $M-U_{n+1}$, and so are nullhomologous in $M-U_n$.  Thus, there exist embedded surfaces $A_i$ and $B_i$ in $M-U_n$ bounded by $\alpha_i$ and $\beta_i$.  Putting them in general position we may assume that $A_i$ and $B_i$ intersect any $\alpha_j$ and $\beta_j$ transversely in a finite number of points.   Since $\lk(\alpha_i, \beta_i) = 1$, the algebraic count of intersections points between $A_i$ and $\beta_i$ is $1$.  By increasing the genus of $A_i$ we may eliminate cancelling intersection points so that $A_i$ intersects $\beta_i$ transversely in a single point.  
The exact same argument allows us to assume that $B_i$ intersects $\alpha_i$ transversely in a single point, and that $A_i\cap \alpha_j = A_i\cap \beta_j = B_i\cap \alpha_j = B_i\cap \beta_j = \emptyset$ for all $i\neq j$.  These surfaces live in $M-U_n$ and $\pi_1(M-U_n)\subseteq \pi_1(M-L)^{(n)}$.  Thus, $\{\alpha_i, \beta_i:i=1,\dots, k\}$ satisifes Definition \ref{defn: solvable surgery} and so Proposition \ref{prop: solvable surgery} concludes that $(M,L)$ is $n$-solvably concordant to $(M_{\alpha, \beta}, L_{\alpha, \beta})$.  Since $(M_{\alpha, \beta}, L_{\alpha, \beta})$ is homology concordant to $(S^3, L')$ we conclude that $(M,L) \simeq_n (S^3, L')$, completing the proof of the theorem.
\end{proof}

\section{Application: bijective satellite operators.}\label{sect:app}

A {pattern} $P$ is a knot embedded in the solid torus $V=S^1\times D^2$.  Patterns act on  knots in $S^3$ via the satellite construction as follows.  Given a knot $K$ in $S^3$ and a pattern $P\subseteq V$, glue together $E(K)$ and $V$ so that the meridian of $K$ is identified with the meridian of $V$ and the $0$-framed longitude of $K$ is identified with the longitude  of $V$.  $P(K)$ is the image of $P$ in this construction.  See for example \cite[Section 4D]{Rolfsen}.
A pattern $P$ which is homologous to $w$ times the prefered generator of $H_1(V)=\Z$ is said to have winding number $w$.  

%

  The rule $K\mapsto P(K)$ induces well defined maps $P:\C\to \C$ and $P:\C_{\smooth}\to\C_{\smooth}$.   The goal of this section is to ask when $P:\C\to \C$ is a bijection.  An easy argument based on the Levine-Tristram signature \cite[Proposition 3.1]{DaRa2017} reveals that no pattern of winding number other than $\pm1$ has any hope of giving a surjection on $\C$ or $\C_{\smooth}$.   Winding number $\pm1$ patterns are more subtle.  In \cite{Levine2016} a winding number $1$ pattern is produced which does not give a surjection on $\C_{\smooth}$.     This section can be thought of as evidence that no such pattern exists for topological concordance, so that every $P:\C\to \C$ is a bijection.  In proving Theorem \ref{thm:application} we see that every winding number $\pm1$ satellite operator acts bijectively on $\C/\F_n$ for all $n\in \N$.

A set of instructions identical to those used to define $P(K)$ can be used to produce a monoid structure on the set of patterns.  For details see \cite[Section 2]{DaRa2017}.  The satellite operation now becomes an action by this monoid.  The submonoid consisting of patterns with winding number $\pm1$ is denoted $\S_\Z$.  In the main theorem of \cite{DaRa2017} appears a group $\widehat \S$ which acts on $\widehat \C$ together with a monoid homomorphism $E:\S_\Z\to\widehat \S$ making the following diagram commute:
$$
   \begin{tikzcd}
\C_\Z \arrow{r}{P} \arrow[hook]{d}{\Psi}
& \C_\Z \arrow[hook]{d}{\Psi} \\
\widehat{\C} \arrow[hook, two heads]{r}{E(P)}
& \widehat{\C}.
   \end{tikzcd}
   $$ 
    Recall that $\C_\Z = \C/\ker(\Psi)$ is the integral knot concordance group.  Since $E(P):\widehat\C\to \widehat\C$ comes from a  group action, it is a bijection.  An easy diagram chase now reveals  that  $P:\C_\Z\to \C_\Z$ is injective.    Indeed, if $\Psi$ were surjective then $P$ would be as well.  
    
  We call the action of $\widehat\S$ on $\widehat\C$ the \emph{generalized satellite construction}.  It is well known that the map $P:\C_\Z\to \C_\Z$ passes to a well defined map $P:\C/\F_n\to \C/\F_n$. The main technical result of this section is that the same is true of the generalized satellite construction.  
    
    \begin{proposition}\label{prop:application}
    If $Q\in \widehat{\S}$ then the generalized satellite construction $Q:\widehat\C\to \widehat\C$ gives a well defined map $Q:\widehat\C/\widehat\F_n\to \widehat\C/\widehat\F_n$.
    \end{proposition}
    
    \begin{proof}[proof of Theorem \ref{thm:application} assuming Proposition \ref{prop:application}]It is an immediate consequence of Proposition \ref{prop:application} that the group action of  $\widehat \S$ on $\widehat \C$ passes to a group action on $\widehat \C/\widehat\F_n$.  Let $P$ be a winding number $\pm1$ pattern and consider the following commutative diagram:
    $$
   \begin{tikzcd}
\C/\F_n \arrow{r}{P} \arrow[hook, two heads]{d}{\Psi}
&\C/\F_n \arrow[hook, two heads]{d}{\Psi} \\
\widehat{\C}/\widehat\F_n \arrow[hook, two heads]{r}{E(P)}
& \widehat{\C}/\widehat\F_n.
   \end{tikzcd}
   $$ 
   Here $\Psi:\C/\F_n\to \widehat\C/\widehat\F_n$ is bijective by Theorem \ref{thm:mainKnot} and $E(P):\widehat{\C}/\widehat\F_n\to \widehat{\C}/\widehat\F_n$ is bijective since it comes from a group action.  It follows immediately  that $P:\C/\F_n\to \C/\F_n$ is bijective, proving Theorem \ref{thm:application}. 
   \end{proof}

Before we prove Proposition \ref{prop:application} we must recall the definition of $\widehat \S$ and its action on $\widehat\C$.  

\begin{definition}[Definition 2.7 of \cite{DaRa2017} with $R=\Z$]
A \emph{generalized pattern} is a triple $(X,i_+,i_-)$ where 
\begin{enumerate}
\item $X$ is an oriented, compact, connected 3--manifold.
\item For $\epsilon\in \{+,-\}$, $i_\epsilon:S^1\times S^1\to \bdry X$ is an embedding and $\bdry X$ = $i_+(S^1\times S^1) \sqcup i_-(S^1\times S^1)$.
\item  $i_+$ is orientation-preserving and $i_-$ is orientation-reversing.
\item $(i_\epsilon)_*:H_*(S^1\times S^1)\to H_*(X)$ is an isomorphism.
\item $(i_+)_*^{-1}\circ(i_-)_*:H_1(S^1\times S^1)\to H_1(S^1\times S^1)$ is $\pm \Id$ where $\Id$ is the identity homomorphism.
\end{enumerate}
\end{definition} 

Notice that the fist four conditions above give a homology cylinder over $S^1\times S^1$.   The quotient of the set of homology cylinders by homology cobordism is a group introduced by J.~Levine and Garoufalidis in \cite{Le8, GarL1}.   This group is denoted by $\mathcal{H}$.  The subgroup of $\mathcal H$ consisting of generalized patterns is denoted $\mathcal{\widehat S}$.  Let $P$ be a winding number $\pm1$ pattern.  Let $i_+:S^1\times S^1 \to \bdry (S^1\times D^2)$ be the natural inclusion, and $i_-:S^1\times S^1\to \bdry \nu(P)$ be the map sending $S^1\times\{pt\}$ to the preferred longitude of $P$ and $\{pt\}\times S^1$ to the meridian.  Then $(E(P), i_+, i_-)$ forms a generalized pattern \cite[Proposition 2.8]{DaRa2017}. 

Next we recall the action of $\widehat\S$ on $\widehat \C$ as presented in \cite[Section 2.5]{DaRa2017}.  Let $(M,K)$ be a knot in a homology sphere and $(X, i_+, i_-)$ be a generalized pattern.  The meridian and longitude of $S^1\times S^1$ are given by $\ell=S^1\times \{pt\}$ and $m = \{pt\}\times S^1$.  Then $(X, i_+, i_-)\cdot (M,K) = (M', K')$ where $M'$ is the $3$-manifold defined by gluing together $E(K)$, $X$, and $S^1\times D^2$ as follows:  First glue $i_+(S^1\times S^1)\subseteq \bdry X$ to $\bdry E(K)$ so that the meridian of $K$ is identified to $i_+(m)$ and the longitude of $K$ is identified to $i_+(\ell)$.  Next glue $i_-(S^1\times S^1)\subseteq \bdry X$ to $\bdry (S^1\times D^2)$ so that $i_-(m)$ is identified with $\{pt\}\times \bdry D^2$ and $i_-(\ell)$ is identified with $S^1\times \{pt\}$.  Using that $(X, i_+, i_-)\in \widehat S$, a direct Mayer-Veitoris argument may be used to show that that $M'$ is a homology sphere.   $K'\subseteq M'$ is the image of core of $S^1\times D^2$ in this construction.  It is checked in \cite[Proposition 2.14]{DaRa2017} that this is compatible with the classical satellite construction in that for any pattern $P$ and any knot $K$ in $S^3$, $E(P)\cdot(S^3,K) = (S^3,P(K))$.  Here equality means orientation preserving homeomorphism of pairs.

We are new ready to check that the generalized satellite construction is compatible with the solvable filtration.  The proof is inspired by the proof of \cite[Proposition 2.15]{DaRa2017}.

\begin{proof}[Proof of Proposition \ref{prop:application}]
Since the quotients $\widehat \C/\widehat \F_n$ and $\widehat \C/ \simeq_n$ are identical, it suffices to prove that the action is well defined on $\widehat \C/ \simeq_n$.  Let $(M,K)\simeq_n(N,J)$ and $(W,C)$ be an $n$-solvable concordance between $(M,K)$ and $(N,J)$.    Suppose that $Q = (X, i_+, i_-)\in \widehat S$.  Set $(M',K') = (X, i_+, i_-)\cdot (M,K)$ and $(N',J') = (X, i_+, i_-)\cdot (N,J)$.    We must build an $n$-solvable concordance $(W',C')$ between $(M',K')$ and $(N',J')$.  Let $E(C)$ be the exterior of $C$.  The gluing instructions used to build $N'$ and $M'$ extend to gluing instructions for a 4--manifold
\begin{eqnarray*}
W'&=&(S^1\times D^2\times[0,1])\cup( X\times[0,1])\cup E(C)
\end{eqnarray*}
in which we see a concordance, $C' = S^1\times\{pt\}\times[0,1]\subseteq S^1\times D^2\times[0,1]\subseteq W'$ from $K'$ to $J'$.  Notice that the copy of $S^1\times D^2\times[0,1]$ above is a tubular neighborhood of $C'$ so that $E(C') = X\times[0,1]\cup E(C)$. 

 A quick Mayer-Veitoris argument reveals that the inclusion induced map $H_*(E(C))\to H_*(E(C'))$ is an isomorphism and that there exists an isomorphism $H_*(W)\to H_*(W')$ making the following diagram commute
\begin{equation}\label{diag:comm}
\begin{tikzcd}
H_*(E(C)) \arrow["\cong",r]\arrow[d] &H_*(E(C'))\arrow[d]\\
H_*(W) \arrow["\cong",r]& H_*(W').
   \end{tikzcd}
\end{equation}
  Let $G=\pi_1(E(C))/\pi_1(E(C))^{(n)}$,  $G' = \pi_1(E(C'))/\pi_1(E(C'))^{(n)}$ and $\widetilde{E(C)}_n$ and $\widetilde{E(C')}_n$ be the induced covers.   By the functoriality of the derived series, the inclusion induced map $\iota_*:\pi_1(E(C))\to \pi_1(E(C'))$ satisfies $\iota_*[\pi_1(E(C))^{(n)}]\subseteq \pi_1(E(C'))^{(n)}$ so that we get a lift
\begin{equation}   \label{diag:lift}
\begin{tikzcd}
\widetilde{E(C)}_n \arrow[dashed]{r}{\widetilde{\iota}} \arrow{d}
&\widetilde{E(C')}_n \arrow{d}\\
E(C) \arrow{r}{\iota}
& E(C').
   \end{tikzcd}
\end{equation}
Let $x_1,y_1, \dots, x_k, y_k\in H_2(E(C);\Z[G]) = H_2(\widetilde{E(C)}_n)$ be the classes guaranteed by Condition \pref{equiv classes} of Definition \ref{defn:n-solvableEquiv}.  Let $x_i'=\widetilde{\iota}_*(x_i)$, and $y_i'=\widetilde{\iota}_*(y_i)$.  By the functoriality of intersection forms, $\lambda_n^{E(C')}(x_i',y_j') = \iota_\sharp\left(\lambda^{E(C)}_n(x_i, y_j)\right) = \delta_{i,j}$, where $\iota_\sharp:\Z[G]\to \Z[G']$ is induced by $\iota_*$.  Similarly, $\lambda^{E(C')}_n(x_i',x_j') = \lambda^{E(C')}_n(y_i',y_j')=0$.

We make use of the commutativity of \pref{diag:comm} and \pref{diag:lift} to conclude that since the image $\{x_i,y_i\}$ in $H_2(E(C);\Z[G])\to H_2(E(C))\to H_2(W)$ gives a basis for $H_2(W)$, it follows that the image of $\{x_i',y_i'\}$ in $H_2(E(C');\Z[G'])\to H_2(E(C'))\to H_2(W')$ gives a basis for $H_2(W')$.  Thus, $(W',C')$ is an $n$-solvable concordance and $(X, i_+, i_-)\cdot (M,K)\simeq_n (X, i_+, i_-)\cdot (N,J)$.  

Finally we conclude that  $(M,K)\mapsto(X, i_+, i_-)\cdot (M,K)$ is well defined on $\C/\simeq_n$, completing the proof.
\end{proof}

\bibliographystyle{plain}

\bibliography{biblio}
\end{document}